\documentclass[12pt]{amsart}
\usepackage{amsmath,amsthm,amssymb,amsfonts,enumerate,color}
\usepackage{amsmath, amsthm, amsfonts, amssymb, color}
\usepackage{mathrsfs}
\usepackage{ amsmath, amsthm, amsfonts, amssymb, color}
 \usepackage{mathrsfs}
\usepackage{amsfonts, amsmath}
 \usepackage{amsmath,amstext,amsthm,amssymb,amsxtra}
 \usepackage{txfonts} 
 \allowdisplaybreaks
 \usepackage{pgf,tikz}

\begin{document}




\def\RN {\mathbb{R}^n}
\newcommand{\norm}[1]{\left\Vert#1\right\Vert}
\newcommand{\abs}[1]{\left\vert#1\right\vert}
\newcommand{\set}[1]{\left\{#1\right\}}
\newcommand{\Real}{\mathbb{R}}
\newcommand{\RR}{\mathbb{R}^n}
\newcommand{\supp}{\operatorname{supp}}
\newcommand{\card}{\operatorname{card}}
\renewcommand{\L}{\mathcal{L}}
\renewcommand{\P}{\mathcal{P}}
\newcommand{\T}{\mathcal{T}}
\newcommand{\A}{\mathbb{A}}
\newcommand{\K}{\mathcal{K}}
\renewcommand{\S}{\mathcal{S}}
\newcommand{\blue}[1]{\textcolor{blue}{#1}}
\newcommand{\red}[1]{\textcolor{red}{#1}}
\newcommand{\Id}{\operatorname{I}}

\newtheorem{thm}{Theorem}[section]
\newtheorem{prop}[thm]{Proposition}
\newtheorem{cor}[thm]{Corollary}
\newtheorem{lem}[thm]{Lemma}
\newtheorem{lemma}[thm]{Lemma}
\newtheorem{exams}[thm]{Examples}
\theoremstyle{definition}
\newtheorem{defn}[thm]{Definition}
\newtheorem{rem}[thm]{Remark}

\numberwithin{equation}{section}

\title[ Carleson measures, BMO spaces and balayages]
{  Carleson measures, BMO spaces and balayages associated to\\[1.5pt] Schr\"odinger operators
}

\author[P. Chen, X.T. Duong, J. Li, L. Song and L.X. Yan]{Peng Chen, \ Xuan Thinh Duong, \
 Ji Li, \ Liang Song \ and \ Lixin Yan
}
\address{Peng Chen, Department of Mathematics, Sun Yat-sen (Zhongshan)
University, Guangzhou, 510275, P.R. China}
\email{achenpeng1981@163.com}
\address{Xuan Thinh Duong, Department of Mathematics, Macquarie University, NSW 2109, Australia}
\email{xuan.duong@mq.edu.au}
\address{Ji Li, Department of Mathematics, Macquarie University, NSW, 2109, Australia}
\email{ji.li@mq.edu.au}
\address{Liang Song, Department of Mathematics, Sun Yat-sen (Zhongshan)
University, Guangzhou, 510275, P.R. China}
\email{songl@mail.sysu.edu.cn}
\address{
Lixin Yan, Department of Mathematics, Sun Yat-sen (Zhongshan) University, Guangzhou, 510275, P.R. China}
\email{mcsylx@mail.sysu.edu.cn
}


  \date{\today}

\begin{abstract}     Let $\L$ be a Schr\"odinger operator of the form
$\L=-\Delta+V$ acting on $L^2(\mathbb R^n)$, $n\geq3$, where the nonnegative
potential $V$ belongs to  the reverse H\"older class    $B_q$ for some $q\geq n.$
 Let ${\rm BMO}_{{\mathcal{L}}}(\RR)$ denote the BMO
space  associated to the Schr\"odinger operator $\L$ on $\RR$.
In this article we show that for every  $f\in {\rm BMO}_{\mathcal{L}}(\RR)$ with compact support,
then there exist   $g\in L^{\infty}(\RR)$
 and a finite Carleson measure $\mu$ such that
 $$
 f(x)=g(x) + S_{\mu, {\mathcal P}}(x)
 $$
 with $\|g\|_{\infty} +\||\mu\||_{c}\leq C \|f\|_{{\rm BMO}_{\mathcal{L}}(\RR)},$
 where
 $$
  S_{\mu, {\mathcal P}}=\int_{{\mathbb R}^{n+1}_+}  {\mathcal P}_t(x,y) d\mu(y, t),
  $$
 and
 ${\mathcal P}_t(x,y)$ is the kernel of the Poisson semigroup
 $\{e^{-t\sqrt{\L}}\}_{t> 0} $ on $L^2(\mathbb R^n)$.
 Conversely, if   $\mu$ is a Carleson measure, then   $S_{\mu, {\mathcal P}}$ belongs to
the space ${\rm BMO}_{{\mathcal{L}}}(\RR)$.
 This  extends the  result for the classical John--Nirenberg BMO
 space by Carleson \cite{C} (see also \cite{U,GJ,W}) to the BMO setting associated to Schr\"odinger operators.
 \end{abstract}

\maketitle


\section{Introduction }
\setcounter{equation}{0}

Consider  the Schr\"odinger  operator with the non-negative  potential $V:$
\begin{equation}\label{e1.5}
\L  =-\Delta  +V(x) \ \ \ {\rm on} \ \ L^2(\RR), \ \ \ \  n\geq 3,
\end{equation}
where   $V$ belongs to  the reverse H\"older class    $B_q$  for  some $q\geq n/2$,
which by definition means that   $V\in L^{q}_{\rm loc}(\RR), V\geq 0$, and
there exists a  constant $C>0$ such that   the reverse H\"older inequality
\begin{equation}\label{e1.6}
\left(\frac{1}{\abs{B}}\int_BV(y)^q~dy\right)^{1/q}\leq\frac{C}{\abs{B}}\int_BV(y)~dy
\end{equation}
holds for all   balls $B$ in $\RR.$

 The operator   $\L$ is a self-adjoint
operator on $L^2(\RR)$, and $\L$ generates the  Poisson semigroup $\{e^{-t\sqrt{\L}}\}_{t>0}$ on $L^2(\RR)$.
Since the potential $V$ is non-negative,  the semigroup kernels ${\mathcal P}_t(x,y)$  of the operators $e^{-t\sqrt{\L}}$ satisfy
\begin{eqnarray*}
0\leq {\mathcal P}_t(x,y)\leq p_t(x-y)
\end{eqnarray*}
for all $x,y\in\RR$ and $t>0$, where
\begin{eqnarray}\label{ep}
p_t(x-y)=c_n {t\over (t^2 +|x|^2)^{{n+1\over 2}}}, \ \ \ \ c_n={\Gamma\Big({n+1\over 2}\Big) \over \pi^{(n+1)/2}}
\end{eqnarray}
is the kernel of the classical Poisson semigroup
$\set{{  P}_t}_{t>0}=\{e^{-t\sqrt{-\Delta}}\}_{t>0}$ on $\Real^n$.

Following \cite{DGMTZ},   a locally integrable function $f$ belongs to  BMO$_{{\mathcal{L}}}({\mathbb R}^n)$
whenever there is constant $C\geq 0$ so that
\begin{equation} \label{e1.7}
  {1\over |Q|}\int_{Q}|f(y)- f_{Q}|dy\leq C,
  \end{equation}
  where $f_Q$ stands for the average of $f$ over the cube  $Q$ on $\RR$,
  i.e., $f_Q={1\over |Q|}\int_Q f(y)dy$, and
\begin{equation}\label{e1.8}
      {1\over |Q|}\int_{Q}|f(y)|dy\leq C
\end{equation}
 for every  cubes   $Q$ with  $\ell(Q)>\rho(x)$, where $\ell(Q)$ is the sidelength of $Q$ and $x$ is the centre of $Q$, and
  the function $\rho(x)=\rho(x; V)$   takes the explicit form
\begin{equation}\label{e1.9}
 \rho(x)=\sup \Big{\{}r>0: \ {1\over r^{n-2}} \int_{B(x,r)} V(y)dy \leq 1 \Big{\}}.
\end{equation}
 Throughout the article we assume that $V\not\equiv 0$ so that $0<\rho(x)<\infty$ (see \cite{Shen}).
We define $\|f\|_{{\rm BMO}_{\mathcal{L}}(\RR)}$ to be the smallest
 $C$  in the right hand sides of \eqref{e1.7} and \eqref{e1.8}.
 Because of \eqref{e1.8}, this ${\rm BMO}_{{\mathcal{L}}}(\RR)$ space is in fact a proper subspace
 of  the classical BMO space of John-Nirenberg.

Consider the function
 \begin{eqnarray}\label{e1.10}
 S_{\mu, {\mathcal P}}(x)=\int_{{\mathbb R}^{n+1}_+}  {\mathcal P}_{t}(x,y) d\mu(y, t),
 \end{eqnarray}
 where $\mu$ is a Borel measure on ${\mathbb R}^{n+1}_+$.
It is easy to see that if $\mu$
 is finite then the integral in \eqref{e1.10} -- called the sweep or balayage of $\mu$ with respect ${\mathcal P}$
 -converges absolutely for a.e. $x\in \RR$, and   $\| S_{\mu, \, {\mathcal P}}\|_1\leq C(n)\|\mu\|.$

  The aim of this article is to prove the following result.

 \begin{thm} \label{th1.1} Suppose $V\in B_q$ for some $q\geq n.$ Then we have

\begin{itemize}
\item[(i)]
 If $\mu$ is a Carleson measure, then $S_{\mu, \, {\mathcal P}}\in {\rm BMO}_{\mathcal{L}}(\RR)$
 with $\|S_{\mu, {\mathcal P}}\|_{{\rm BMO}_{\mathcal{L}}(\RR)} \leq C(n)\||\mu\||_c.$

 \item[(ii)]
Let $f\in {\rm BMO}_{\mathcal{L}}(\RR)$ have  compact support. There exist  $g\in L^{\infty}(\RR)$
 and a finite Carleson measure $\mu$ such that
  \begin{eqnarray}\label{e1.11}
 f(x)=g(x) + S_{\mu, {\mathcal P}} (x),
\end{eqnarray}
 where $\|g\|_{L^\infty(\RR)} +\||\mu\||_{c}\leq C(n)\|f\|_{{\rm BMO}_{\mathcal{L}}(\RR)}.$
 \end{itemize}
 \end{thm}

 We would like to mention that for the case that  ${\mathcal P}_t(x,y) $ is the classical Poisson kernel
  $p_t(x-y)$
 in \eqref{ep}, the proof of (i) of Theorem~\ref{th1.1} of the classical BMO is quite standard (see \cite{G});
  and the result  (ii) of Theorem~\ref{th1.1}   is proved
 by an iteration argument in the  work of L. Carleson \cite{C}; and also in the work of A. Uchiyama \cite{U}, and  by Garnett and Jones \cite{GJ}.
 Later,  J.M. Wilson \cite{W} gives a new proof by using the Poisson semigroup property  and Green's theorem
 to avoid the iteration to make the construction much more explicit.
Our Theorem~\ref{th1.1} extends the result   of the classical BMO
 to the space ${\rm BMO}_{\mathcal{L}}(\RR)$ associated with the Schr\"odinger operators.
The proof   of Theorem~\ref{th1.1} follows the idea of \cite{W} i.e., by using the Poisson semigroup property, and Green's theorem,
but differs
 from it in method and techniques since the kernel for the operator $\mathcal{L}$ is not translation
 invariant and several techniques for the classical Possion kernel are not applicable here.
The standard preservation condition  of the semigroup  $\{e^{-t\L}\}_{t>0}$ fails,  i.e.,
\begin{eqnarray*}
e^{-t\L}1 \not=1.
\end{eqnarray*}
  This is indeed one of the  main obstacles in this article and
  makes the theory quite subtle and delicate. We
  overcome  this problem in the proof
  by  making  use of the   estimates
on the kernel of the Poisson semigroup  under the assumption on $V\in B_q$ for some $q\geq n$,
and some techniques to estimate the norm of the space ${\rm BMO}_{\mathcal{L}}(\RR)$
associated with the Schr\"odinger operators.

 Throughout the article, the letter  ``$C$ " will  denote (possibly different) constants
 that are independent of the essential variables.

\medskip

\section{Preliminaries}
 \setcounter{equation}{0}

 Throughout the article,
we may sometimes use capital letters to denote points in
${\Bbb R}^{n+1}_+$, e.g.,
$X=(x, t),$  and set
\begin{eqnarray*}
{ \nabla}_X u(X)=(\nabla_x u, \partial_tu) \ \ \ {\rm and}\ \ \
|{ \nabla}_Xu|=|\nabla_x u| +|\partial_tu|.
\end{eqnarray*}
For simplicity we will denote by
$\nabla$  the full gradient $\nabla_X$ in $\mathbb{R}^{n+1}$.
For every cube $Q$ on $\RR$, we set  $\ell(Q)$   the sidelength of $Q$, and
let $CQ$ denote the cube concentric with $Q$ and with sidelength $C$
times as big. For  $Q$ as above,  we write ${\widehat Q}=\{(x, t)\in {\Bbb R}^{n+1}_+:  x\in Q, 0<t<\ell(Q) \}$.
 Recall that a
measure $\mu$ defined on ${\mathbb R}^{n+1}_+$ is said to  be a
Carleson measure if
\begin{equation}
|||\mu|||_{c}=: \sup_{Q\subset \RR} {|\mu|({\widehat Q})\over |Q|} <\infty.
\label{e1.1}
\end{equation}

 \subsection{Basic properties of the heat and Poisson semigroups of Schr\"odinger operators}

Let us recall  some basic properties of
the critical radii
 function $\rho(x)$ under the assumption \eqref{e1.6} on $V$
(see Section 2, \cite{DGMTZ}).  Suppose $V\in B_q$ for some $q> n/2.$
There exist $C>0$ and $k_0\geq1$ such that for all $x,y\in\Real^n$
\begin{equation}\label{e2.1}
C^{-1}\rho(x)\left(1+\frac{\abs{x-y}}{\rho(x)}\right)^{-k_0}\leq\rho(y)\leq C\rho(x)
\left(1+\frac{\abs{x-y}}{\rho(x)}\right)^{\frac{k_0}{k_0+1}}.
\end{equation}
In particular, $\rho(x)\sim \rho(y)$ when $y\in B(x, r)$ and $r\leq c\rho(x)$.

 The following estimates on the heat kernel of $L$ are well known.
\begin{prop}[\cite{DGMTZ}]\label{prop-heat}
	Let $L=-\Delta+V$ with $V\in B_q$ for some $q\ge n/2$. Then for each $N>0$ there exists $C_N>0$ such that
	the heat kernel $\mathcal{H}_t(x,y)$ of the semigroup $\{e^{-tL}\}_{t>0}$ satisfies
	\begin{align}\label{hk bound}
		\mathcal{H}_t(x,y)\le C_N\frac{e^{-|x-y|^2/ct}}{t^{n/2}} \left({1+\frac{\sqrt{t}}{\rho(x)}+\frac{\sqrt{t}}{\rho(y)}}\right)^{-N}
	\end{align}
	and
	\begin{align}\label{hk holder}
		|\mathcal{H}_t(x,y)-\mathcal{H}_t(x',y)|\le C_N \left({\frac{|x-x'|}{\sqrt{t}}}\right)^{\beta} \frac{e^{-|x-y|^2/ct}}{t^{n/2}} \left({1+\frac{\sqrt{t}}{\rho(x)}+\frac{\sqrt{t}}{\rho(y)}}\right)^{-N}
	\end{align}
	whenever $|x-x'|\le \sqrt{t}$ and for any $0<\beta< \min\{1, 2-n/q\}$.
\end{prop}


%
%
%

Through Bochner's
  subordination formula (see \cite{St}), the Poisson semigroup associated to $\L$ can be obtained
from the heat semigroup:
\begin{align}\label{e3.28}
e^{-t\sqrt{\L}}f(x)=\frac{1}{\sqrt{\pi}}\int_0^\infty\frac{e^{-u}}{\sqrt{u}}~e^{-{t^2\over 4u}\L}f(x)~du=\frac{t}{2\sqrt\pi}\int_0^\infty
\frac{~e^{-{t^2\over 4s}}}{s^{3/2}}e^{-s\L}f(x) ~ds,
\end{align}
which connects the Poisson kernel and the heat kernel as follows:
\begin{align}\label{e PH}
\mathcal{P}_t(x,y)
&=\frac{t}{2\sqrt\pi}\int_0^\infty
\frac{~e^{-{t^2\over 4s}}}{s^{3/2}} \mathcal{H}_s(x,y) ds.
\end{align}

\begin{lem}[\cite{DYZ, DGMTZ, MSTZ}] \label{le2.1}Suppose $V\in B_q$ for some $q\geq n.$ For any $0<\beta< \min\{ 1, 2-n/q\}$ and every $N>0$, there exists
 a constant $C=C_{N}$ such that for $x,y\in \RR$ and $t>0,$
 \begin{itemize}

\item[(i)]
 ${\displaystyle
| {\mathcal P}_t(x,y)|\leq C{t \over (t^2+|x-y|^2)^{n+1\over 2}} \left(1+ {(t^2+|x-y|^2)^{1/2}\over \rho(x)}
+{(t^2+|x-y|^2)^{1/2}\over \rho(y)}\right)^{-N}; }
 $
\smallskip

\item[(ii)]   The kernel $\nabla {\mathcal P}_t(x,y)$ satisfies
 $${\displaystyle
|t\nabla {\mathcal P}_t(x,y)|\leq {C t  \over (t^2+|x-y|^2)^{n+1\over 2}} \left(1+ {(t^2+|x-y|^2)^{1/2}\over \rho(x)}
+{(t^2+|x-y|^2)^{1/2}\over \rho(y)}\right)^{-N};
}
$$

\smallskip

\item[(iii)]
$\displaystyle \big| t\nabla  e^{-t\sqrt{\L}}(1)(x) \big|\le C
  \left(\frac{ {t}}{\rho(x)}\right)^{\beta} \left(1+ {t\over \rho(x)}\right)^{-N}.$
   \end{itemize}

\end{lem}

It follows from Lemmas 1.2 and 1.8 in \cite{Shen}  that there is a constant $C>0$ such that
 for a nonnegative Schwartz class function $\varphi$ there exists a constant $C$ such that
 \begin{equation}\label{e2.2}
 \int_{\RR} \varphi_t(x-y)V(y)dy\leq
 \left\{
 \begin{array}{lll}
 Ct^{-1}\big({\sqrt{t}\over \rho(x)}\big)^{\delta}\ \ \ &{\rm for}\ t\leq \rho(x)^2\, ;\\[8pt]
 C\big({\sqrt{t}\over \rho(x)}\big)^{C_0+2-n}\ \ \ &{\rm for}\ t> \rho(x)^2\, ,
 \end{array}
 \right.
 \end{equation}
 where $\varphi_t(x)=t^{-n/2}\varphi(x/\sqrt{t}),$   and
 $
 \delta =2-n/q>0.
 $
 From \eqref{e2.1}, we can follow the proof of (d) of Proposition 3.6, \cite{MSTZ} to show
 that
for $V\in B_q,q\geq n/2, $
there is some $\delta\in(0,1)$ and $N> \max\{4, 2\log(C_0+2-n)\}$ with $C_0$ the doubling constant of with respect to
the potential $V$ as in  (1.1), such that for $t>0$ and $x\in\mathbb R^n$,
 \begin{eqnarray}\label{e2.3}
  \int_{{\Bbb R}^n} \left(1+ \left |{\rm log} {\rho(y)\over t}\right|\right)
    \, t{\mathcal P}_t(x,y) V(y)dy \le C t^{-1}\min \left\{
  \left(\frac{ {t}}{\rho(x)}\right)^{\delta}, \, \left( {t\over \rho(x)}\right)^{-{N\over2}+2}\right\}.
   \end{eqnarray}

\smallskip

\begin{lem}\label{le2.2} Suppose $V\in B_q$ for some $q\geq n/2.$
Let $f\in {\rm BMO}_{{\mathcal{L}}}({\mathbb R}^n).$ Then we have
 \begin{itemize}

\item[(i)]  There exists $C>0$ such that for all   $Q
 $ with $\ell(Q)<\rho(x)$, then
 $$
 |f_{2Q}|\leq C\left(1+{\rm log} {\rho(x)\over \ell(Q)}\right)  \|f\|_{{\rm BMO}_{\mathcal{L}}(\RR)}.
 $$

\item[(ii)] For every $p\in [1, \infty),$ there exists $C=C(p, \rho)>0$ such that
 $$
 \left({1\over |Q|}\int_Q |f-f_Q|^pdx \right)^{1/p}\leq C  \|f\|_{{\rm BMO}_{\mathcal{L}}(\RR)}, \ \
 \mbox{ for all cubes  $Q$},
 $$
 and for every $Q$ with $\ell(Q)\geq \rho(x):$
  $$
 \left({1\over |Q|}\int_Q |f|^pdx \right)^{1/p}\leq C  \|f\|_{{\rm BMO}_{\mathcal{L}}(\RR)}.
 $$
 \end{itemize}
\end{lem}

\begin{proof}
For the proof of (i), we can obtain it by making minor modification with that of Lemma 2, \cite{DGMTZ},
and we skip it here. For the proof of (ii), we refer it to Corollary 3, \cite{DGMTZ}.
\end{proof}

 \medskip

\subsection{The Balayage associated with generalized approximation to the identity}

In this section we will work with a class of integral operators $\{{\mathcal A}_t\}_{t>0}$, which
plays the role of   generalized approximations to the identity.
Assume that the operators ${\mathcal A}_t$ are defined by kernels
$a_t(x,y)$ in the  sense that
$$
{\mathcal A}_tf(x)=\int_{{\mathbb R^n}}{\mathcal A}_t(x,y)f(y)dy,
$$
where the kernels ${\mathcal A}_t(x,y)$ satisfy the estimate
\begin{eqnarray}\label{e2.4}
|{\mathcal A}_t(x,y)| \leq C t^{-n} \left(1+{|x-y|\over t}\right)^{-(n+\epsilon)}
\end{eqnarray}
for some $\epsilon >0 $ and
for every function $f$ which satisfies some suitable  growth condition.

The space   ${\rm{BMO_{\mathcal A}}}(\RR)$  associated with   a generalized approximation
to the identity $\{A_t\}_{t>0}$ was introduced and studied in \cite{DY1}.
In the sequel,   $\epsilon$ is a constant as in \eqref{e2.4}.

\begin{defn}\label{def2.3}
 We say that $f\in L^1((1+|x|)^{-(n+\epsilon)})$ is in ${\rm{BMO_{\mathcal A}}}(\RR)$,
the space of functions of bounded  mean oscillation associated with a
generalized approximation to
the identity $\{{\mathcal A}_t\}_{t>0}$,
if there exists some constant $C$ such that for any ball $B$,
\begin{equation}
  {1\over |Q|}\int_Q|f(x)-{\mathcal A}_{\ell(Q)}f(x)|dx\leq C,
\label{e2.5}
\end{equation}
where   $\ell(Q)$ is the sidelength  of the cube $Q$.
The smallest bound $c$ for which (\ref{e2.5}) is satisfied is then taken to
  be the norm of $f$ in this space, and is denoted by $\|f\|_{\rm{BMO_{\mathcal A}}}$.
\end{defn}

 Note first  that  (${\rm{BMO_{\mathcal A}}}$, $\|\cdot \|_{\rm{BMO_{\mathcal A}}}$) is a
semi-normed vector space, with the seminorm vanishing on the space $
{\mathcal K}_{\mathcal A}$, defined by
$$
{\mathcal K}_{\mathcal A}=\bigg\{f\in L^1((1+|x|)^{-(n+\epsilon)}): {\mathcal A}_tf(x)=f(x) \ {\rm for}\
 {\rm  almost\
all\ }\ x\in {\RR} {\rm \ for\ all}\ t>0 \bigg\}.
$$
The space ${\rm{BMO_{\mathcal A}}}$   is understood to be modulo $
{\mathcal K}_{\mathcal A}$. It is easy to check that
$L^{\infty}(\RR)\!\subset {\rm{BMO_{\mathcal A}}}(\RR)$    with $
\|f\|_{\rm{BMO_{\mathcal A}}}\leq C\|f\|_{L^{\infty}(\RR)}$.

\medskip

In this section, we assume   that

 i) $A_0$ is the identity operator and the operators $\{{\mathcal A}_t\}_{t>0}$
form a semigroup, that is, for any $t, s>0$ and $f\in {\mathcal M}$,
${\mathcal A}_{t}{\mathcal A}_{s}f(x)={\mathcal A}_{t+s}f(x)$
for almost all $x\in{\RR}$.

ii) There exists some  $0<\epsilon'\leq \epsilon$ such that
the kernels ${\mathcal A}_t$ satisfy the estimate
\begin{eqnarray}\label{e2.6}
  \left| t{\partial\over \partial t} {\mathcal A}_t(x,y)\right|
  \leq C t^{-n} \left(1+{|x-y|\over t}\right)^{-(n+\epsilon')}.
\end{eqnarray}

Examples of operators $\{{\mathcal A}_t\}_{t>0}$  which satisfy conditions i) and ii) above include
the Poisson and heat kernels of certain    operators including
 Schr\"odinger operators   with nonnegative potentials  and second order divergence form elliptic operators
   (see for example, \cite{DY1, DY2,  DYZ,  DGMTZ, MSTZ, Shen, St}).

 \begin{prop} \label{prop2.4} Assume that $\{{\mathcal A}_t\}_{t>0}$ is  a generalized approximation
to the identity with properties i) and ii) above.
 If $\mu$ is a Carleson measure on ${\Bbb R}^{n+1}_+$, then the function
 \begin{eqnarray}\label{e2.7}
 S_{\mu, \, {\mathcal A}}(x)=\int_{{\mathbb R}^{n+1}_+}  {\mathcal A}_t(x,y) d\mu(y, t)
 \end{eqnarray}
 belongs to $  {\rm BMO}_{\mathcal{A}}(\RR)$
 with $\|S_{\mu, {\mathcal A}}\|_{{\rm BMO}_{\mathcal{A}}(\RR)} \leq C(n)\||\mu\||_c.$
 \end{prop}

 \begin{proof} Let $Q$ be a cube with the center $x_Q$ and its sidelength $\ell(Q)$. Let $Q_k=2^k Q$.
It follows from the assumption (i) of  $\{{\mathcal A}_t\}_{t>0}$ that
 \begin{align} \label{e2.8}
 \int_{Q}|S_{\mu, {\mathcal A}}(x)-{\mathcal A}_{\ell(Q)}S_{\mu, {\mathcal A}}(x)|dx
 &\leq   \int_{Q}\int_{{\mathbb R}^{n+1}_+}  |{\mathcal A}_t(x,z)- {\mathcal A}_{\ell(Q)}
 {\mathcal A}_t(x,z)| d\mu(z, t)dx\\
  &\leq  \int_{Q}\int_{{\widehat Q_{1}}}  |{\mathcal A}_t(x,z)- {\mathcal A}_{t+\ell(Q)}(x,z)| d\mu(z, t)dx\nonumber\\
 &\quad+\int_{Q}  \int_{\mathbb R^{n+1}_+\backslash  {\widehat Q_{1}} }
 |{\mathcal A}_t(x,z)- {\mathcal A}_{t+\ell(Q)} (x,z)| d\mu(z, t)dx=I+II.\nonumber
\end{align}
For the term $I,$ one easily sees that
 \begin{eqnarray} \label{e2.9}
   I&=& \int_{{\widehat Q_{1}}}  \int_{Q} |{\mathcal A}_t(x,z)- {\mathcal A}_{t+\ell(Q)}(x,z)|
   dx  d\mu(z, t)
 \leq  C|{\widehat Q_{1}}|_{\mu}\leq C |Q|.
\end{eqnarray}

Consider the term $II.$ We apply the formula:
 \begin{eqnarray*}
{\mathcal A}_t(x,z)- {\mathcal A}_{t+\ell(Q)}(x,z)=-\int_0^{\ell(Q)} \partial_s {\mathcal A}_{s+t}(x,z)ds.
\end{eqnarray*}
  This, together with the assumption  (ii) of  $\{{\mathcal A}_t\}_{t>0}$, yields
  that for some $0<\epsilon'\leq \epsilon,$
 \begin{align} \label{e2.10}
 II&\leq \int_Q\int_0^{\ell(Q)} \int_{\mathbb R^{n+1}_+\backslash  {\widehat Q_{1}} }
 |\partial_s {\mathcal A}_{s+t}(x,z)|\, d\mu(z, t) ds dx\\
 &= \sum_{k=1}^{\infty}  \int_Q \int_0^{\ell(Q)} \int_{{\widehat Q_{k+1}} \backslash
{\widehat Q_{k}} } |\partial_s {\mathcal A}_{s+t}(x,z)|\, d\mu(z, t) ds dx\nonumber\\
&\leq  C\sum_{k=1}^{\infty}  \int_Q\int_0^{\ell(Q)} \int_{{\widehat Q_{k+1}} \backslash
{\widehat Q_{k}} } {1\over s+t} {(s+t)^{\epsilon'} \over (s+t+ |x-z|)^{n+\epsilon'}}
 \, d\mu(z, t) ds dx\nonumber\\
 &=:  \sum_{k=1}^{\infty} II_k.\nonumber
\end{align}
Notice that for every $x\in Q$, $0<t<\ell(Q)$ and $(s, z)\in
{\widehat Q_{k+1}} \backslash {\widehat Q_{k}}, k=1, 2, \cdots $, we have that
$
s+t+ |x-z|\geq 2^{k}\ell(Q).
$
It tells us that
 \begin{align*}
 II_k &\leq C (2^k\ell(Q))^{-(n+\epsilon')}
 \int_{Q}\int_0^{\ell(Q)} \int_{{\widehat Q_{k+1}} \backslash
{\widehat Q_{k}} }   {(s+t)^{\epsilon'-1}  }
 \, d\mu(z, t) ds dx\nonumber\\
 &\leq C |Q| (2^k\ell(Q))^{-(n+\epsilon')} \ell(Q)^{\epsilon'} |2^kQ|\nonumber\\
  &\leq 2^{-k\epsilon'}|Q|,
\end{align*}
which shows that $II\leq \sum_{k=1}^{\infty}2^{-k\epsilon'}|Q| \leq C|Q|.$ This, in combination with
\eqref{e2.9}, yields the desired result. This completes the proof.
\end{proof}

%

 \medskip

\section{Proof of Theorem~\ref{th1.1}}
 \setcounter{equation}{0}


Let ${\rm BMO}_{\mathcal P}(\RR)$ denote the BMO space associated with the semigroup
$\{e^{-t\sqrt{\L}}\}_{t>0}$ defined  in Definition~\ref{def2.3}.
Under the assumption of $V\in B_q$ for some $q\geq n/2, $  it is known (see \cite[Proposition 6.1]{DY2})
that the spaces ${\rm BMO}_{\mathcal L}(\RR)$ and
${\rm BMO}_{\mathcal P}(\RR)$ coincide, and their norms are equivalent.
From this,  (i) of Theorem~\ref{th1.1}
is a straightforward consequence of
Proposition~\ref{prop2.4}.



We now show  (ii) of Theorem~\ref{th1.1}.
Throughout this section,  we will  assume that $ f\in {\rm BMO_\L}(\RR) $  with
 ${\rm supp} f \subseteq \{ x= (x_1, \cdots, x_n):\  \ |x_i|\leq 1, i=1, \cdots, n\}$ and set $ Q_0=\{ x= (x_1, \cdots, x_n):\  \ |x_i|\leq 2, i=1, \cdots, n\}$.
 Write
$$u(x, t)={\mathcal P}_tf(x) = e^{-t\sqrt{\L}} f(x).
$$
 We will use two facts about the space
${\rm BMO}_{\L}(\RR)$(see \cite{DYZ, DGMTZ, MSTZ})  under the assumption of $V\in B_q$, $q\geq n$:

(1) Let $Q=Q(x, t)$ denote the cube with center $x$ and sidelength  $t$. Then we have
  \begin{eqnarray}\label{e3.1}
  |  {\mathcal P}_t(f-f_Q)(x) | =  |  e^{-t\sqrt{\L}}(f-f_Q)(x) | \leq C\|f\|_{{\rm BMO_\L}(\RR)}.
\end{eqnarray}

(2) Let $\nabla$ denote the full gradient in ${\Bbb R}^{n+1}_+$. We have  that for all $(x,t)\in {\Bbb R}^{n+1}_+,$
\begin{eqnarray}\label{e3.2}
 \left | t\nabla u(x, t)\right | = \left | t\nabla e^{-t\sqrt{\L}}(f)(x)\right |\leq C \|f\|_{{\rm BMO_\L}(\RR)}.
\end{eqnarray}

In the sequel, for a cube $Q\subset \Bbb R^n$ we let $z_Q=(x_Q, t_Q)$
with $t_Q=\ell(Q)$,
where $x_Q$ is the center of $Q$ ($z_Q$ is the center of the top face of ${\hat Q}$).
 We define generations, ${\bf G}_k,$ of subcubes of $Q_0$
as follows:
\begin{align*}
{\bf G}_0&=\{Q_0\};\\
{\bf G}_{k+1}&=\left\{Q'\subseteq Q\in {\bf G}_k: \ Q'  \  \mbox{ maximal dyadic such that }\ \left|
u(x_Q, t_Q)
-u(x_{Q'}, t_{Q'})\right| >A \right\}
 \end{align*}
for $k\geq 0$, where $A$ is a large constant to be chosen later.

In the sequel, for every $Q\in {\bf G}_k, k\geq 0 $  set
\begin{eqnarray}\label{e3.9}
 \Sigma_Q={\widehat Q}\  {\Big \backslash}
\bigcup_{\substack{Q'\subset Q\in {\bf G}_{k}\\ Q'\in {\bf G}_{k+1}} } {\widehat Q}',
\end{eqnarray}
and  $\partial\Sigma_Q$ denotes the boundary of $\Sigma_Q.$
Then the following result holds.

\begin{lemma}\label{le3.1} We have the properties:

\begin{itemize}
\item[(i)] There exists a constant $C>0$ such that for every $(x,t)\in  \Sigma_Q\cup \partial \Sigma_Q,$
$$
|u(x, t)-u(x_Q, t_Q) |\leq A +C\, \|f\|_{{\rm BMO_\L}(\RR)}.
$$

\item[(ii)] For every $Q\in {\bf G}_k, k\geq 0,$
$$
\sum_{\substack{Q'\subset Q\in {\bf G}_{k}\\ Q'\in {\bf G}_{k+1}} } |Q'|\leq CA^{-1}|Q|\, \|f\|_{{\rm BMO_\L}(\RR)}.
$$
As a consequence,  if  $A$ is large enough, then
$\sum_{\substack{Q'\subset Q\in {\bf G}_{k}\\ Q'\in {\bf G}_{k+1}} } |Q'|\leq
 |Q|/2$.

\end{itemize}
\end{lemma}

\begin{proof} For a cube $J$, we set $T(J)$ the top half of ${\widehat J}, $
i.e.,
$
T(J)=\{ (x,t):  x\in J, \ell(J)/2\leq t<\ell(J)\}.
$
We notice that, for every dyadic cube $I,$
$$
{\widehat I}=\bigcup_{\substack{J\subset I \\ J\in {\mathcal D} } } T(J),
$$
and therefore, if $Q\in {\bf G}_k, \Sigma_Q$ is the union of all the sets $T(J)$ such that
$J\subseteq Q$ is dyadic and $J$ is not a subset of any $Q'\in {\bf G}_{k+1}.$
If $(x,t)\in \Sigma_Q,$ then $(x,t)$ lies in some $T(J)$ as described above. But then
$|u(x_Q, t_Q) - u(x_J, t_J) |\leq A$ because if it were
$|u(x_Q, t_Q) - u(x_J, t_J)|> A$, $J$ would belongs to
${\bf G}_{k+1}$ or would be contained in some cube in ${\bf G}_{k+1}$, i.e., a maximal
cube for the property. But then $(x,t)$ would not line in $\Sigma_Q.$ On the other hand,
 the fact \eqref{e3.2} implies  that for every $(x,t)\in T(J),$
$$
|u(x, t) - u(x_J, t_J) |\leq C\|f\|_{{\rm BMO_\L}(\RR)}.
$$
All together, yields the desired result for $(x,t)\in \Sigma_Q$. Taking limits, this also holds for $
(x, t)\in \partial \Sigma_Q$. This proves  (i).

Let us show (ii).  For every $k,$ it follows from the definition of ${\bf G}_{k}$
  that
\begin{eqnarray}\label{e3.4}
 \sum_{\substack{Q'\subset Q\in {\bf G}_{k}\\ Q'\in {\bf G}_{k+1}} } |Q'|  \leq   {1\over A}
\sum_{\substack{Q'\subset Q\in {\bf G}_{k}\\ Q'\in {\bf G}_{k+1}} }
  |Q'| |u(x_Q, t_Q)  - u(x_{Q'}, t_{Q'})|
 \end{eqnarray}
Observe  that $  u(x_Q, t_Q)={\mathcal P}_{t_Q}f(x_Q)=
 {\mathcal P}_{t_Q}(f-f_Q)(x_Q)+ {\mathcal P}_{t_Q} (1)(x_{Q})  f_{Q}$. We obtain
\begin{eqnarray*}
\mbox{RHS of \eqref{e3.4}}
&\leq&   {1\over A}
\sum_{\substack{Q'\subset Q\in {\bf G}_{k}\\ Q'\in {\bf G}_{k+1}} }
  |Q'|\, \Big(| {\mathcal P}_{t_Q}(f-f_Q)(x_Q)|+| {\mathcal P}_{t_{Q'}}(f-f_{Q'})(x_{Q'})| + |{\mathcal P}_{t_{Q'}} (1)(x_{Q'})|  |f_Q-f_{Q'}|\Big)\nonumber
\\&&\quad + {1\over A}
\sum_{\substack{Q'\subset Q\in {\bf G}_{k}\\ Q'\in {\bf G}_{k+1}} }
  |Q'|\, |f_Q| \Big|     {\mathcal P}_{t_Q} (1)(x_{Q})  -     {\mathcal P}_{t_{Q'}} (1)(x_{Q})  \Big|\nonumber
\\&&\quad + {1\over A}
\sum_{\substack{Q'\subset Q\in {\bf G}_{k}\\ Q'\in {\bf G}_{k+1}} }
 |Q'|\,   |f_Q| \,  | {\mathcal P}_{t_{Q'}} (1)(x_Q)- {\mathcal P}_{t_{Q'}} (1)(x_{Q'}) |\nonumber
\\
 &=&: I+II+III. \nonumber
  \end{eqnarray*}

We first note that
$I\leq  {C\over A} |Q|\, \|f\|_{{\rm BMO}_{\mathcal{L}}(\RR)}$. In fact, this follows from the estimate in \eqref{e3.1}
and from the facts that for every $t>0$ and $x\in\mathbb R^n$, $|{\mathcal P}_{t} (1)(x)|\leq C$ and that for every $Q'\subset Q$, $|f_Q -f_{Q'}|\leq |Q'|^{-1}\int_{Q'}|f -f_{Q}|dx \leq C \|f\|_{{\rm BMO}_{\mathcal{L}}(\RR)}$.

To estimate  the term $II,$  we first
 assume   that $t_Q\geq  \rho(x_Q)$. Then from   the facts that  that for every $t>0$ and $x\in\mathbb R^n$, $|{\mathcal P}_{t} (1)(x)|\leq C$ and that $\abs{f_{Q}}\le C \norm{f}_{\rm BMO_\L(\RR)}$,
the term $II$ is bounded by
 \begin{eqnarray}\label{e3.5}
 {1\over A}
\sum_{\substack{Q'\subset Q\in {\bf G}_{k}\\ Q'\in {\bf G}_{k+1}} }
  |Q'|\,  \norm{f}_{\rm BMO_\L(\RR)} \leq  {C\over A} |Q|\, \|f\|_{{\rm BMO_\L}(\RR)},
 \end{eqnarray}
where the last inequality follows from the fact that these $Q'$'s are pairwise disjoint.

Next we consider the case  $t_Q<\rho(x_Q)$. From \eqref{e2.1}, we have that $\rho(x)
\sim \rho(x_Q)$ for $x\in Q$.
  By (iii) of Lemma~\ref{le2.1},
 \begin{align*}
|{\mathcal P}_{t_Q} (1) (x_{Q}) - {\mathcal P}_{t_{Q'}}(1) (x_{Q}) |
 &=\left|\int_{t_{Q'}}^{t_Q} s\partial_s e^{-s\sqrt{\L}} (1)(x_{Q}) {ds\over s}\right|\\
 &\leq    C \int_{t_{Q'}}^{t_Q}  \left( {s\over \rho(x_Q)}\right)^{\beta} {ds\over s}
 \\
 &\leq         C\left( {t_Q\over \rho(x_Q)}\right)^{\beta},
 \end{align*}
  which, together with (i) of Lemma~\ref{le2.2}, implies
 \begin{align}\label{e3.6}
\sum_{\substack{Q'\subset Q\in {\bf G}_{k}\\ Q'\in {\bf G}_{k+1}} } {1\over A}
 |Q'|\, \,   |f_Q| \,  |  {\mathcal P}_{t_Q} (1) (x_Q) - {\mathcal P}_{t_{Q'}}(1) (x_{Q})   |
&\leq {C\over A} \left( {t_Q\over \rho(x_Q)}\right)^{\beta} \abs{f_{Q}}
\sum_{\substack{Q'\subset Q\in {\bf G}_{k}\\ Q'\in {\bf G}_{k+1}} } \int_{Q'}  dy\\
&\leq
 {C\over A} |Q|\,  \abs{f_{Q}}    \left( {t_Q\over \rho(x_Q)}\right)^{\beta}\nonumber\\
 &\leq
 {C\over A} |Q|\, \|f\|_{{\rm BMO_\L}(\RR)}
 \Big(1+\log {\rho(x_Q)\over t_Q}\Big)   \left( {t_Q\over \rho(x_Q)}\right)^{\beta}\nonumber\\
  &\leq
 {C\over A} |Q|\, \|f\|_{{\rm BMO_\L}(\RR)}.\nonumber
\end{align}

For the term $III$, we follow an argument as that in the term $II$ and
 the fact \eqref{e3.2}  to show that in this case,
 $ III\leq   {C\over A} |Q|\, \|f\|_{{\rm BMO_\L}(\RR)}.
 $
Therefore,
\begin{eqnarray}\label{emm}
\sum_{\substack{Q'\subset Q\in {\bf G}_{k}\\ Q'\in {\bf G}_{k+1}} } |Q'|
\leq {C\over A} |Q|\, \|f\|_{{\rm BMO_\L}(\RR)} \leq {1\over 2}|Q|,
\end{eqnarray}
by choosing $A$  large enough. This ends the proof of (ii) of Lemma~\ref{le3.1}.  The proof   is complete.
\end{proof}

In the sequel, we fix a constant $A>0$ large enough so that \eqref{emm} holds.
 By (ii) of Lemma~\ref{le2.2}, we know that $f\in L^2_{\rm loc}(\RR)$. Since $\supp f\subset Q_0=\{ x= (x_1, \cdots, x_n):\  \ |x_i|\leq 2, i=1, \cdots, n\},$
 $f$ belongs to $L^2(\RR).$ It follows by the spectral theory (\cite{St})
   that
   \begin{eqnarray}\label{e3.7}
f(x)=  2\int_0^{\infty} \sqrt{\L} e^{-2t\sqrt{\L}} f(x)dt,
\end{eqnarray}
where the improper integral converges in $L^2(\RR)$.
Note that both functions ${\mathcal P}_t(x, y)$ and $ u(y,t)$ are $C^1(\RR)$ in $y$. Since $-u_{tt} +\L u=0,$ we apply
  Green's theorem to obtain
\begin{align*}
0&= -
\iint_{{\mathbb R}^{n+1}_+} t{d^2\over dt^2} {\mathcal P}_t(x,y) u(y,t) dydt +
 \iint_{{\mathbb R}^{n+1}_+} t {\mathcal P}_t(x,y) (-\Delta +V) u(y,t) dydt\\
&=
\iint_{{\mathbb R}^{n+1}_+}  {d\over dt} {\mathcal P}_t(x,y)u(y,t) dydt +
 \iint_{{\mathbb R}^{n+1}_+} t \nabla {\mathcal P}_t(x,y)\  \nabla u(y,t) dydt\\
 &\quad+
 \iint_{{\mathbb R}^{n+1}_+} t {\mathcal P}_t(x,y)  V(y) u(y,t)  dydt.
\end{align*}
By \eqref{e3.7}, it follows that
\begin{eqnarray}\label{e3.8}
 f(x)= 2\iint_{{\mathbb R}^{n+1}_+} \left( t \nabla {\mathcal P}_t(x,y) \nabla u(y,t)
 + t {\mathcal P}_t(x,y)  V(y) u(y,t) \right)dydt,
\end{eqnarray}
where the improper integral converges in $L^2(\RR)$.

We are now going to cut up the integral in \eqref{e3.7}. For every $Q\in {\bf G}_k, k\geq 0,$
we recall that $\Sigma_Q$ is defined in \eqref{e3.9}.
Observe  that if $Q\not =\cap Q',$ then $\Sigma_Q\cap \Sigma_{Q'}=\emptyset$. It follows that
\begin{eqnarray} \label{e3.10}
{\mathbb R}^{n+1}_+ =  {\widehat Q}_0  \cup   \left({\mathbb R}^{n+1}_+\backslash {\widehat Q}_0\right)
 = \left( \bigcup_{k\geq 0} \bigcup_{Q\in {\bf G}_k} \Sigma_{Q} \right) \bigcup
  \left({\mathbb R}^{n+1}_+\backslash {\widehat Q}_0\right).
\end{eqnarray}
From this, we rewrite
\begin{eqnarray}\label{e3.11}
 f(x)&= &  \sum_{k\geq 0} \sum_{Q\in {\bf G}_k}2\iint_{\Sigma_{Q}} \left( t \nabla {\mathcal P}_t(x,y) \nabla u(y,t)
 + t {\mathcal P}_t(x,y)  V(y) u(y,t) \right) dydt\\
& & + 2\iint_{{\mathbb R}^{n+1}_+\backslash {\widehat Q}_0} \left( t \nabla {\mathcal P}_t(x,y) \nabla u(y,t)
 + t {\mathcal P}_t(x,y)  V(y) u(y,t) \right)dydt\nonumber\\
&=:& \left( \sum_{k\geq 0} \sum_{Q\in {\bf G}_k} f_Q(x)\right) + f_2(x)
 =:    f_1(x) + f_2(x).\nonumber
\end{eqnarray}

We will deal with $f_1(x)$ first.  Since $-u_{tt} +\L u=0,$ we
 apply Green's theorem to each $f_Q$ of the summands in $f_1$ to obtain
\begin{align}\label{e3.12}
f_Q(x)
&= u(y_Q,t_Q)  \iint_{\Sigma_Q} t {\mathcal P}_t(x,y)
  V(y)     dydt\nonumber\\
   &\quad +\left\{
 \iint_{\partial \Sigma_Q} t {\partial \over \partial \nu }  {\mathcal P}_t(x,y)  \,
  \left(u(y,t)-u(y_Q,t_Q)\right) d\sigma_Q\right.
  +\iint_{\partial \Sigma_Q}  t   {\mathcal P}_t(x,y) \,  {\partial \over \partial \nu }  u(y,t)     d\sigma_Q \\
 &\quad\left. -\iint_{\partial \Sigma_Q} {\mathcal P}_t(x,y) \,
\left(u(y,t)-u(y_Q, t_Q) \right)   {\partial t \over \partial \nu }d\sigma_Q\right\}\nonumber\\
  &=: I_Q(x) +II_Q(x). \nonumber
 \end{align}

\medskip

\begin{lemma}\label{3.2}  There exists a constant $C>0$ such that for all $x\in{\Bbb R}^n,$
$$
  \sum_Q   | I_Q(x)|   \leq CA+ C \|f\|_{{\rm BMO_\L}(\RR)}.
$$
\end{lemma}

\begin{proof}   To estimate the term
$$
 \iint_{\Sigma_Q} u(y_Q,t_Q) t {\mathcal P}_t(x,y)
  V(y)     dydt.
  $$
We
write $u(y_Q, t_Q)={\mathcal P}_{t_Q}f(y_Q)={\mathcal P}_{t_Q}(f-f_{Q(y, \, t_Q)})(y_Q)
+{\mathcal P}_{t_Q}(1)(y_Q) f_{Q(y, \, t_Q)}$ where $(y,t)\in \Sigma_Q$.
  This, in combination with \eqref{e3.1} and \eqref{e2.3},
 shows   that
\begin{align}\label{e3.14}
\sum_Q | I_Q(x)|&\leq  \sum_Q \iint_{\Sigma_Q}  |{\mathcal P}_{t_Q}(f-f_{Q(y,\, t_Q)})(y_Q)|  t {\mathcal P}_t(x,y)
  V(y)     dydt\\
  &\quad+  \sum_Q \iint_{\Sigma_Q}  |{\mathcal P}_{t_Q} (1)(y_Q) f_{Q(y,\, t_Q)}|  t {\mathcal P}_t(x,y)
  V(y)     dydt\nonumber\\
  &\leq C \|f\|_{{\rm BMO_\L}(\RR)} + C\sum_Q {\widetilde {I_Q}}(x) .  \nonumber
  \end{align}
  By Lemma~\ref{le2.2},
    \begin{eqnarray*}
  |f_{Q(y, \, t_Q)}| \leq
  \left\{
  \begin{array}{lll}
  C\|f\|_{{\rm BMO_\L}(\RR)}
   \left(1+{\rm log} {\rho(y)\over t_Q}\right), &\ \ {\rm if}& \ \ t_Q\leq \rho(y); \\[8pt]
    C\|f\|_{{\rm BMO_\L}(\RR)}, & \ \ {\rm if}& \ \ t_Q> \rho(y).
   \end{array}
   \right.
   \end{eqnarray*}
  It follows from \eqref{e2.3} that
   \begin{align*}
 \sum_Q {\widetilde {I_Q}}(x)
  &\leq
  C \|f\|_{{\rm BMO_\L}(\RR)}
   \sum_Q \iint_{\Sigma_Q} \left(1+ \max \left\{ {\rm log} {\rho(y)\over t_Q}, 1\right\} \right) t {\mathcal P}_t(x,y) V(y)dydt \\
  &\leq    C \|f\|_{{\rm BMO_\L}(\RR)}
  \iint_{\Bbb R^{n+1}_+} \left(1+ \left |{\rm log} {\rho(y)\over t}\right|\right) t {\mathcal P}_t(x,y)  V(y)    dy{dt }  \\
  &\leq     C \|f\|_{{\rm BMO_\L}(\RR)}  \int_0^{\infty} \min\left\{ \left({t\over \rho(x)}\right)^{\delta}, \left({t\over \rho(x)}\right)^{-{N\over2}+2}  \right\} {dt\over t} \\
  &\leq   C \|f\|_{{\rm BMO_\L}(\RR)},
  \end{align*}
  which, together with \eqref{e3.14}, shows that
  $\sum_Q | I_Q(x)|  \leq C \|f\|_{{\rm BMO_\L}(\RR)}.$
  The proof is complete.
\end{proof}

\medskip
Next we estimate the term    $II_Q(x)$.
Following \cite{W}, one writes
$$
F(Q: t,x,y)= t\left( {\mathcal P}_t(x,y) \, {\partial \over \partial \nu }  u(y,t)
 +
{\partial \over \partial \nu } {\mathcal P}_t(x,y) \left(u(y, t)-u(y_Q,t_Q)\right)  \!  \right) \!- {\mathcal P}_t(x,y)
\left(u(y,t)-u(y_Q,t_Q)\right) {\partial t \over \partial \nu }
$$
and
\begin{align}\label{e3.144}
 II_Q(x)
&=\!\!\iint_{\partial \Sigma_Q}\!F(Q: t,x,y) d\sigma_Q\nonumber\\
&=   \int_{\partial \Sigma_Q\cap {\Bbb R}^n} F(Q: t,x,y)    d\sigma_Q +\int_{\partial \Sigma_Q\cap \{ t>0\}} F(Q: t,x,y)    d\sigma_Q
  =:  II^{(1)}_Q(x)  +II^{(2)}_Q(x).
\end{align}
It can be verified that
$II^{(1)}_Q(x)$ is equal to
$$
\left(f(x)-u(y_Q,t_Q)\right) \chi_{\partial\Sigma_Q\cap {\Bbb R}^n}(x)=h_Q(x).
$$
The supports of the different ${h_Q}^{,}$s  are easily seen to be disjoint, and so we may set
\begin{eqnarray*}
h(x)=\sum_{Q\in \big(\bigcup_{k=0}^{\infty}{\bf G}_k\big)} h_Q(x).
\end{eqnarray*}
Then the following result holds.

\begin{lemma}\label{le3.3}  The function $h$ satisfies
\begin{eqnarray}\label{e3.15}
\|h\|_{\infty}\leq A + C\|f\|_{{\rm BMO_\L}(\RR)}
\end{eqnarray}
 with supp $h\subseteq \{x=(x_1, x_2, \cdots, x_n): \ |x_i|\leq 2, i=1, 2, \cdots, n\}.$
\end{lemma}

\begin{proof} By (ii) of Lemma~\ref{le2.2}, we know that $f\in L^2_{\rm loc}(\RR)$. Since $f\in {\rm BMO}_{\L}(\RR)$
have compact support, $f$ belongs to $L^2(\RR).$
Then by  (b) of Maximal Theorem on Stein's book \cite[Section 3 of Chapter III]{St},
$$
\lim\limits_{t\to  0}e^{-t\sqrt{\L}}  f(x)=f(x), \ \ \  {\rm a.e.}
$$
By (i) of Lemma~\ref{le3.1}, it follows that  for every $(x,t)\in  \Sigma_Q,$
$$
|e^{-t_Q\sqrt{\L}} (f)(x_Q)-e^{-t\sqrt{\L}} (f)(x) |\leq A+ C\|f\|_{{\rm BMO_\L}(\RR)}.
$$
Taking limits,  this holds for each $h_Q(x), x\in \partial\Sigma_Q\cap {\Bbb R}^n.$  This proves
\eqref{e3.15}.
\end{proof}

We will finished with the term $f_1$ once we show that
\begin{eqnarray}\label{e3.16}
\sum_Q II^{(2)}_Q(x)= S_{\mu, {\mathcal P}}
\end{eqnarray}
for some $\mu$ with $\||\mu\||_c\leq C(n) \|f\|_{{\rm BMO_\L}(\RR)}.$ We proceed to do this now. Write
\begin{align}\label{e3.1444}
II^{(2)}_Q(x)&= \iint_{\partial \Sigma_Q \cap \{ t>0\}} t\left( {\mathcal P}_t(x,y) \, {\partial \over \partial \nu }  u(y,t)
 +
{\partial \over \partial \nu } {\mathcal P}_t(x,y)\, \left(u(y, t)-u(y_Q,t_Q)\right)  \,   \right)d\sigma_Q\nonumber\\
&\quad- \iint_{\partial \Sigma_Q \cap \{ t>0\}}
  {\mathcal P}_t(x,y)
\left(u(y,t)-u(y_Q,t_Q)\right) {\partial t \over \partial \nu }
  d\sigma_Q\nonumber\\
  &=:II^{(21)}_Q(x)+II^{(22)}_Q(x).
\end{align}
By fact \eqref{e3.2} and the way we chose the $Q'{\rm s}$,
$$
\left|t {\partial \over \partial \nu }  u(y,t)  -\left(u(y,t)-u(y_Q,t_Q)\right) {\partial t \over \partial \nu }
\right|\leq C(n) \|f\|_{{\rm BMO_\L}(\RR)}
$$
when $(y,t)\in \partial\Sigma_Q\cap\{t>0\}$. Hence,
\begin{eqnarray*}
\iint_{\partial \Sigma_Q \cap \{ t>0\}}  \left(
t {\partial \over \partial \nu }  u(y,t)  -\left(u(y,t)-u(y_Q,t_Q)\right) {\partial t \over \partial \nu } \right)
{\mathcal P}_t(x,y) d\sigma_Q
 =\iint_{\partial \Sigma_Q \cap \{ t>0\}} {\mathcal P}_t(x,y) h_Q(y,t)d\sigma_Q
\end{eqnarray*}
for some $h_Q(y,t)$ with $|h_Q|\leq C(n) \|f\|_{{\rm BMO_\L}(\RR)}$.
By (ii) of Lemma~\ref{le3.1}, it is well known (see \cite[Chapter VIII]{G}, \cite{W}) that
\begin{eqnarray}\label{e3.17}
 \left\|\left|\sum_{k\geq 0}\sum_{Q\in{\bf G}_k}  d\sigma_Q\right\|\right|_c\leq C(n) ,
\end{eqnarray}
which, together with the condition $|h_Q|\leq C(n)\|f\|_{{\rm BMO_\L}(\RR)}$, implies that
\begin{eqnarray*}
 \left\|\left|\sum_{k\geq 0}\sum_{Q\in{\bf G}_k}  h_Q(y,t) d\sigma_Q\right\|\right|_c\leq C(n) \|f\|_{{\rm BMO_\L}(\RR)}.
\end{eqnarray*}
Thus, to obtain \eqref{e3.16}, we only need to  estimate the integrals
 \begin{eqnarray}\label{e3.18}
 \iint_{\partial \Sigma_Q \cap \{ t>0\}} t
{\partial \over \partial \nu } {\mathcal P}_t(x,y)\, \left(u(y, t)-u(y_Q,t_Q)\right)  \, d\sigma_Q.
\end{eqnarray}
For this, we need  the following lemma.

\begin{lemma}\label{le3.4}  Let $\{z_i\}=\{(x_i, t_i)\}\subset {\mathbb R}^{n+1}_+$
be points and let $\delta_{z_i}=$ the Dirac mass at $z_i$. Assume that
$$
\left\|\left| \sum_i   t_{i}^n \delta_{z_{i}} \right\|\right|_c \leq 1.
$$
Let $ds_i$ denote $n$-dimensional Lebesgue measure on the hyperplane $\{ t=t_i\}\subset {\mathbb R}^{n+1}_+$
and let $\Phi_t(x, y)=t/(t+|x-y|)^{n+1}$. Set

$$
\mu(x,t)=\sum_i t_{i}^n \Phi_{t_i} (x, x_i) ds_i.
$$
Then $\||\mu\||_c\leq C(n).$
\end{lemma}

\begin{proof}
Its proof is similar to that of LEMMA of \cite{W}, and we skip it here.
\end{proof}

To estimate \eqref{e3.18},  we follow \cite{W} to write
$$
\partial \Sigma_Q \cap \{ t>0\}=\bigcup_i E_{i, Q}
$$
where each $E_{i, Q}$ is of the form
\begin{eqnarray*}
E_{i, Q}=
\left\{
\begin{array}{ll}
\{(x, t): \ x\in \partial Q_i, &\ \ell(Q_i) \leq t\leq 2\ell(Q_i)\}\cap \partial \Sigma_Q\ \ \mbox{or}\\[6pt]
\{(x, t): \ x\in {\overline Q_i}, &\ t=\ell(Q_i)  \}
\end{array}
\right.
\end{eqnarray*}
for some dyadic cube $Q_i.$ These $E_{i, Q}$ make a tiling of $\partial \Sigma_Q \cap \{ t>0\}$,
with the size of the tiles going to zero as $t\to 0.$  Let $(x_{i, Q}, {\tilde  t}_{i, Q})$ be the centroid
of $E_{i, Q}$ in ${\mathbb R}^{n+1}_+$ and let $\xi_{i, Q}=(x_{i, Q}, \, {1\over 2} {\tilde t}_{i, Q} )
\equiv (x_{i,Q}, t_{i, Q}).$
It is easy to see that, for any cube $Q^{\ast}=2Q$,
\begin{eqnarray*}
\int_{{\widehat {Q^{\ast}}}} \sum_i {\tilde  t}_{i, Q}^n \delta_{\xi_{i, Q}} dydt\leq
\int_{ 4{\widehat {Q^{\ast}}}} d\sigma_Q.
\end{eqnarray*}
Therefore,
by \eqref{e3.17},
\begin{eqnarray}\label{e3.19}
\left\|\left| \sum_{k\geq 0}\sum_{Q\in{\bf G}_k} \sum_i {\widetilde  t}_{i, Q}^n \, \delta_{\xi_{i, Q}} \right\|\right|_c \leq C(n) \|f\|_{{\rm BMO_\L}(\RR)}.
\end{eqnarray}
Observe that
\begin{eqnarray*}
t
{\partial \over \partial \nu } {\mathcal P}_t(x,y)
=\int_{\Bbb R^n} {\mathcal P}_{t_{i, Q}}(x, z) \left(t
{\partial \over \partial \nu } {\mathcal P}_{t-t_{i, Q}} (z, y)\right)  dz.
\end{eqnarray*}
For every $(y, t)\in E_{i, Q}$,  it is easy to see that $|y-x_{i, Q}|\leq t.$
Also we have that   $t- t_{i, Q}> t_{i, Q}/3 $,
$t\leq 8t_{i, Q}/3$, and so $t\sim t- t_{i, Q}.$ Then there exists a constant  $C>0$
uniformly for $(y, t)\in E_{i, Q}$ such that
$$
 \left|t
{\partial \over \partial \nu } {\mathcal P}_{t-t_{i, Q}} (z, y)\right|
\leq C {t_{i, Q} \over \left(t_{i, Q}+ |z- x_{i, Q}|\right)^{n+1}}.
$$
It then follows that
\begin{eqnarray}\label{e3.20}
\sum_{k\geq 0}\sum_{Q\in{\bf G}_k} II^{(22)}_Q(x)
 &=& \sum_{k\geq 0}\sum_{Q\in{\bf G}_k}\sum_{i } \iint_{E_{i, Q}} t
{\partial \over \partial \nu } {\mathcal P}_t(x,y)\, \left(u(y, t)-u(y_Q,t_Q)\right)  \, d\sigma_Q\\
&=&  \sum_{k\geq 0}\sum_{Q\in{\bf G}_k}\sum_{i } t^n_{i, Q}   \int_{\Bbb R^n} {\mathcal P}_{t_{i, Q}}(x, z)
 \Lambda^{i,Q}_{t-t_{i, Q}}  (z, x_{i, Q})dz_{i, Q}\nonumber\\
 &=&  \int_{{\mathbb R}^{n+1}_+} {\mathcal P}_t(x, y) d\rho_Q(y, t),\nonumber
\end{eqnarray}
where
\begin{eqnarray}\label{e3.21}
d\rho_Q= \sum_{k\geq 0}\sum_{Q\in{\bf G}_k} \sum_{i} t^n_{i, Q} \Lambda^{i,Q}_{t-t_{i, Q}}  (z, x_{i, Q})dz_{i, Q}
\end{eqnarray}
and
 \begin{eqnarray}\label{e3.22}
|\Lambda^{i,Q}_{t-t_{i, Q}}(z, x_{i, Q})|&=& \left|t^{-n}_{i, Q}
 \iint_{E_{i, Q}}   \left(u(y, t)-u(y_Q,t_Q)\right)  \left(t
{\partial \over \partial \nu } {\mathcal P}_{t-t_{i, Q}} (z, y)\right)  \, d\sigma_Q\right|\\
&\leq&  C {t_{i, Q} \over \left(t_{i, Q}+ |z- x_{i, Q}|\right)^{n+1}}.\nonumber
\end{eqnarray}
Here we used the fact  that for every $(y, t)\in E_{i, Q}\subset \partial \Sigma_Q$,
it follows by (i) of Lemma~\ref{le3.1} that
$|u(y, t)-u(y_Q,t_Q)\leq A+ C\|f\|_{{\rm BMO_\L}(\RR)}.$
Hence, from estimates \eqref{e3.19}, \eqref{e3.20}, \eqref{e3.21} and \eqref{e3.22},
 we apply   Lemma~\ref{le3.4} to  finish  the proof of the term $f_1.$

Now for the term $f_2(x)$ in \eqref{e3.11},  Green's theorem gives us
\begin{align}\label{e3.23}
f_2(x)&= u(y_{Q_0},t_{Q_0})  \iint_{{{\mathbb R}^{n+1}_+\backslash {\widehat Q_0}}} t {\mathcal P}_t(x,y)
  V(y)     dydt\nonumber\\
  & +\left\{\iint_{\partial  {\widehat Q_0} \cap {\mathbb R}^{n+1}_+ }
  t {\partial \over \partial \nu }  {\mathcal P}_t(x,y)  \,  \left(u(y,t)-u(y_{Q_0},t_{Q_0})\right) d\sigma_{Q_0}\right.\\
 &\quad+  \iint_{\partial  {\widehat Q_0} \cap {\mathbb R}^{n+1}_+ }  t   {\mathcal P}_t(x,y) \,  {\partial \over \partial \nu }  u(y,t)
   d\sigma_{Q_0} \nonumber\\
 &\left.\quad- \iint_{\partial  {\widehat Q_0} \cap {\mathbb R}^{n+1}_+ } {\mathcal P}_t(x,y) \,
\left(u(y,t)-u(y_{Q_0}, t_{Q_0}) \right)   {\partial t \over \partial \nu }d\sigma_{Q_0}\right\}\nonumber\\
  &=III_{Q_0}(x) +IV_{Q_0}(x), \nonumber
 \end{align}
where $d\sigma_{Q_0}$ is $n$-dimensional surface measure on $\partial {\widehat Q_0} $
and $\partial/\partial\nu$ now denotes the normal derivative into ${\widehat Q_0}$.

To estimate the term $III_{Q_0}(x),$ we write $ u(y_{Q_0},t_{Q_0})= ( u(y_{Q_0},t_{Q_0})-f_{Q_0}) +f_{Q_0}$, it follows
by Lemma~\ref{le2.2} that $|u(y_{Q_0},t_{Q_0})|\leq  C \|f\|_{{\rm BMO_\L}(\RR)}.$ We apply an argument as in $I_Q(x)$ to show that
\begin{eqnarray*}
|III_{Q_0}(x)|
  &\leq  &    C \|f\|_{{\rm BMO_\L}(\RR)}
  \iint_{\Bbb R^{n+1}_+}   t {\mathcal P}_t(x,y)  V(y)    dy{dt }  \\
 &\leq&     C \|f\|_{{\rm BMO_\L}(\RR)}  \int_0^{\infty} \min\left\{ \left({t\over \rho(x)}\right)^{\delta}, \left({t\over \rho(x)}\right)^{-{N\over2}+2}  \right\} {dt\over t} \\
 &\leq&    C \|f\|_{{\rm BMO_\L}(\RR)}.
\end{eqnarray*}
For the term  $IV_{Q_0}(x)$, we notice that ${\partial {\widehat Q_0} \cap {\mathbb R}^{n+1}_+}$ is away from the support of $f$,
$$
|u(y,t)-u(y_{Q_0}, t_{Q_0})|\leq C \|f\|_{{\rm BMO_\L}(\RR)}\ \ \ \ {\rm and}\ \  |t\nabla u(y, t)|\leq C \|f\|_{{\rm BMO_\L}(\RR)}
$$
on
${\partial {\widehat Q_0} \cap {\mathbb R}^{n+1}_+} $. So, since $d\sigma_{Q_0}$ is a Carleson measure
with norm $\leq C$, these terms   present no problem. We now handle the other term   by cutting
${\partial {\widehat Q_0}}$ into tiles, just as we did for $f_1(x)$, to  obtain   estimate for $IV_{Q_0}(x)$. This finishes the proof
of      $f_2(x)$.

 Finally, we  collect estimates \eqref{e3.11}, \eqref{e3.12}, \eqref{e3.144} for   $f_1(x)$ and  \eqref{e3.23} for  $f_2(x)$ to write
 $$f=g(x) + S_{\mu, {\mathcal P}} (x)
$$
where
$$
g(x) =\sum_{k\geq 0} \sum_{Q\in {\bf G}_k}  I_Q(x) + \sum_{k\geq 0} \sum_{Q\in {\bf G}_k}  II_Q^{(1)}(x)  +III_{Q_0}(x) \in L^{\infty}(\RR)
$$
and a finite Carleson measure $\mu$  such that
$$
S_{\mu, {\mathcal P}} (x) =\sum_{k\geq 0} \sum_{Q\in {\bf G}_k}II_{Q}^{(2)}(x) +IV_{Q_0}(x)=\int_{{\mathbb R}^{n+1}_+}  {\mathcal P}_{t}(x,y) d\mu(y, t)
$$
with $  \|g\|_{L^\infty(\RR)}+ \||\mu\||_{c}\leq C(n)\|f\|_{{\rm BMO}_{\mathcal{L}}(\RR)}$.
The proof of  Theorem~\ref{th1.1} is complete.

\bigskip

 Consider $\L=-\Delta +V(x)$, where $V\in L^1_{\rm loc}({\Bbb R}^n)$ is a
non-negative function on ${\Bbb R}^n$.
Let  $\{e^{-t\L}\}_{t>0}$ be the heat  semigroup associated to $\L$:
\begin{eqnarray}\label{e3.24}
 e^{-t\L}f(x)=\int_{\Real^n}{\mathcal H}_t(x,y)f(y)~dy,\qquad f\in L^2(\Real^n),~x\in\Real^n,~t>0.
\end{eqnarray}
Since the potential $V$ is nonnegative,  the    kernel ${\mathcal H}_t(x,y)$
of the semigroup
$e^{-t\L}$  satisfies
\begin{eqnarray}
0\leq {\mathcal H}_t(x,y)\leq h_t(x-y)
\label{e3.25}
\end{eqnarray}
for all $x,y\in\RR$ and $t>0$, where
\begin{equation}\label{e3.26}
h_t(x-y)=\frac{1}{(4\pi t)^{n/2}}~e^{-\frac{\abs{x-y}^2}{4t}}
\end{equation}
is the kernel of the classical heat semigroup
$\set{T_t}_{t>0}=\{e^{t\Delta}\}_{t>0}$ on $\Real^n$.
Then we have the following result.

 \begin{thm} \label{prop3.5} Suppose $V\in B_q$ for some $q\geq n.$ Then  for
 every $f\in {\rm BMO}_{\mathcal{L}}(\RR)$ with compact support, there exist   $g\in L^{\infty}(\RR)$
 with compact support and a finite Carleson measure $\mu$ such that
 $$
 f(x)=g(x) + \int_{{\mathbb R}^{n+1}_+}  {\mathcal H}_{t^2}(x,y) d\mu(y, t),
 $$
 where $\|g\|_{L^\infty(\RR)} +\||\mu\||_{c}\leq C(n)\|f\|_{{\rm BMO}_{\mathcal{L}}(\RR)}.$
 \end{thm}

 \begin{proof}
From Theorem~\ref{th1.1}, we know that for
 every $f\in {\rm BMO}_{\mathcal{L}}(\RR)$ with compact support, there exist $g\in L^{\infty}(\RR)$
 and a finite Carleson measure $\mu$ such that
 $
 f(x)=g(x) + S_{\mu, {\mathcal P}} (x),
 $
 where $\|g\|_{L^\infty(\RR)} +\|\mu\|_{C}\leq C(n)\|f\|_{{\rm BMO}_{\mathcal{L}}(\RR)}.$
 From this, and the subordination formula \eqref{e3.28}, we rewrite
\begin{align*}
S_{\mu, {\mathcal P}}(x)&=\int_{{\mathbb R}^{n+1}_+}  {\mathcal P}_{t}(x,y) d\mu(y, t) \\
&=\int_{{\mathbb R}^{n+1}_+}  {\mathcal H}_{s^2}(x,y)  \left(  \frac{1}{ \sqrt\pi}  \int_0^\infty  t
\frac{~e^{-{t^2\over 4s^2}}}{s^{2}}   d\mu(y, t)  ~ds \right)\\
&=:\int_{{\mathbb R}^{n+1}_+}  {\mathcal H}_{s^2}(x,y)  d\mu(y, s).
\end{align*}
The proof reduces to show that
$$
d\mu(y, s)=:\frac{1}{ \sqrt\pi}  \int_0^\infty  t
\frac{~e^{-{t^2\over 4s^2}}}{s^{2}}   d\mu(y, t)  ~ds
$$
is a Carleson measure on ${\mathbb R}^{n+1}_+$. Indeed,  for each cube  $Q$ on ${\mathbb R}^{n}$,
\begin{align*}
\int_0^{\ell(Q)}\!\!\int_Q \left(   \int_0^\infty  t
\frac{~e^{-{t^2\over 4s^2}}}{s^{2}}  d\mu(y, t)  ~ds \right)
&=\int_0^{\ell(Q)}\int_Q \left(   \int_0^{\ell(Q)} t
\frac{~e^{-{t^2\over 4s^2}}}{s^{2}}  d\mu(y, t)  ~ds \right)\\
&\quad+\sum_{k=1}^{\infty} \int_0^{\ell(Q)}\int_Q \left(  \int_{2^{k-1}\ell(Q)}^{2^k\ell(Q)}  t
\frac{~e^{-{t^2\over 4s^2}}}{s^{2}}  d\mu(y, t)  ~ds \right)\\
&\leq  \int_{\widehat Q}  d\mu(y, t) \left(   \int_0^{\infty}
\frac{~e^{-{ 1\over 4s^2}}}{s^{2}} ~ds \right)
+ \sum_{k=1}^{\infty} \int_{\widehat {2^kQ}} d\mu(y, t)\left(   \int_0^{2^{-k} }
\frac{~e^{-{ 1\over 4s^2}}}{s^{2}} ~ds \right)  \\
&\leq  C|Q| + C\sum_{k=1}^{\infty} 2^{-k(n+1)}|2^kQ|\\
&\leq C|Q|.
\end{align*}
This completes the proof.
\end{proof}

Finally, we apply Theorem~\ref{th1.1} to discuss   the dual theory of the spaces    $H^1_{\L}(\RR)$  and  ${{\rm BMO}_{\L}(\RR)}$ associated
to the Schr\"odinger operator. Let    $\L=-\Delta +V$    with $V\in B_q$ for some $q\geq n.$
Recall that a  Hardy-type space associated to $\L$ was
introduced by J. Dziubanski et  al.(see \cite{DGMTZ}),  defined  by
\begin{equation}\label{e4.1}
 H^1_{\L}(\RR):=\big\{ f\in L^1(\RR): {\mathcal P}^{\ast}f(x)= \sup_{|x-y|\leq t}|e^{-t\sqrt{\L}}f(y)|\in L^1(\RR) \big\}
\end{equation}
with
$
\|f\|_{H^1_{\L}(\RR)}:=\|{\mathcal P}^{\ast}f\|_{L^1(\RR)}.
$

 For such class of potentials,   $H^1_{\L}(\RR)$ admits an atomic characterization, where
cancellation conditions are only required for atoms with small supports.
It is known that if  $V\in B_q$ for some $q> n/2,$ then
 the dual space of $H^1_{\L}(\RR)$ is  ${{\rm BMO}_{\L}(\RR)}$, i.e.,
\begin{eqnarray}\label{e4.22}
(H^1_{\L}(\RR))^{\ast}={{\rm BMO}_{\L}(\RR)}.
\end{eqnarray}
 The proof of \eqref{e4.22} was given in \cite[Theorem 4]{DGMTZ}. See also \cite{DY2}.
Now, we can  derive the  half of the duality result \eqref{e4.22} from Theorem~\ref{th1.1}, see the proposition below.
We note that our proof here is independent of atomic decomposition of the Hardy space $H^1_{\L}(\RR)$ (\cite[Theorem 6.2]{HLMMY}).

\begin{prop}\label{prop4.1} Suppose $V\in B_q$ for some $q\geq  n.$ Then
${{\rm BMO}_{\L}(\RR)}$ is in
 the dual space of $H^1_{\L}(\RR)$, i.e.,
$$
{{\rm BMO}_{\L}(\RR)}\subseteq  (H^1_{\L}(\RR))^{\ast}.
$$
\end{prop}

\begin{proof} Let $g\in {\rm BMO}_{\L}({\Bbb R^n})$ with compact support
and $\|g\|_{{{\rm BMO}_{\L}(\RR)}}= 1$.
By Theorem~\ref{th1.1}, there   exist  $h\in L^{\infty}(\RR)$
 and a finite Carleson measure $\mu$ such that
 $
g(x)=h(x) + S_{\mu, {\mathcal P}} (x),
 $
 where $\|h\|_{L^\infty(\RR)} +\||\mu\||_{c}\leq C(n).$
 Then for every $f\in H^1_{\L}(\RR),$
\begin{eqnarray*}
 \left| \int_{\RR} f(x) g(x)dx\right|&\leq &  \left| \int_{\RR} f(x) h(x) dx\right|
 + \left| \int_{\mathbb R^{n} } f(x)  S_{\mu, {\mathcal P}} (x) dx \right| \\
 &\leq &  C\|f\|_{L^1(\mathbb R^n)}
 + \left| \int_{\mathbb R^{n+1}_+} {\mathcal P}_tf(x)  d\mu(x,t) \right|\\
  \\
 &\leq &  C\|f\|_{L^1(\mathbb R^n)}
 + C\|{\mathcal P}^{\ast}f\|_{L^1(\mathbb R^n)}  \\
 &\leq & C\|f\|_{H^1_{\L}(\mathbb R^n)}.
\end{eqnarray*}
The desired result  follows by the standard density argument.
\end{proof}

\medskip
\noindent
{\bf Remarks.}
We would like to comment on the possibility of several generalizations and
open problems related to  Theorem~\ref{th1.1}.

 (1)  The first one is the extension of  (ii) of  Theorem~\ref{th1.1}  for the Schr\"odinger operators
 $-\Delta+V$ with  the  nonnegative potential $V\in B_q$ for some $q\geq n/2$,
 or assuming merely  that  the   potential
  $V$ of $\L$   is a locally nonnegative integrable function on ${\Bbb R}^n$,
  and this question will be considered in the future.

(2)
The  proof of  Theorem~\ref{th1.1}  uses the Poisson semigroup property, and Green's theorem in $\mathbb R^n$.
  We may ask whether  (ii) of Theorem~\ref{th1.1} still holds
for the  space ${\rm BMO}_{\L}(X)$ associated to the Poisson semigroup $\{ e^{-t\sqrt{\L}}\}$
 of abstract selfadjoint
operators $\L$ on
 metric measure space $X$
with certain proper assumptions, along which direction
there have been  lots of success
in the last few years,
see for example, in  \cite{DY1, DY2, DYZ, DGMTZ, HLMMY, MSTZ, Shen} and the
references therein.

\vskip 1cm

\noindent
{\bf Acknowledgments.} L. Yan would like to thank J. Michael Wilson for helpful discussions.
C. Peng is supported by the NNSF
of China, Grant No.~11501583. X. T. Duong  is supported by
Australian Research Council  Discovery Grant DP 140100649.
J. Li is supported by ARC DP 170101060.
L. Song is supported in part by the NNSF of China (Nos 11471338 and 11622113)
and Guangdong Natural Science Funds for Distinguished Young Scholar (No. 2016A030306040).
L. Yan is supported by the NNSF
of China, Grant No.~11371378 and  ~11521101  and Guangdong Special Support Program.

 \bigskip




\begin{thebibliography}{10}







\bibitem{C} L. Carleson, Two remarks on $H^1$ and BMO, {\it Adv. Math.} {\bf 22} (1976), 269-277.




\bibitem   {DY1}  X.T. Duong and L.X. Yan, New function spaces of
{\rm BMO} type, the John-Nirenberg inequality, interpolation and applications,
{\it  Comm. Pure Appl. Math.} {\bf 58} (2005), 1375--1420.


\bibitem {DY2} X.T. Duong and L.X. Yan, Duality of Hardy and BMO spaces
associated with operators with heat kernel bounds. {\it J. Amer. Math. Soc.}
{\bf 18}(2005), 943--973.

 \bibitem {DYZ} X.T. Duong, L.X. Yan and C. Zhang,
  On characterization of Poisson integrals  of Schr\"odinger operators with BMO
 traces. {\it J. Funct. Anal.}{\bf  266}  (2014),   2053--2085.







\bibitem{DGMTZ} J. Dziuba\'nski, G. Garrig\'os, T. Mart\'inez, J. L. Torrea and J. Zienkiewicz,
 BMO  spaces related to Schr\"odinger operators with potentials satisfying a reverse H\"older inequality.
\textit{Math. Z.}
\textbf{249} (2005), 329--356.




 \bibitem{G} J. Garnett, {\it Bouned analytic functions}, Academic Press, New York,1981.


\bibitem{GJ} J. Garnett and P. Jones, BMO from dyadic BMO, {\it Pacific J. Math.} {\bf 99} (1982), 351-371.


\bibitem{Gra} L. Grafakos, Classical Fourier Analysis. 2nd edition. GTM, {\bf249}, Springer,
New York, 2008.


\bibitem{HLMMY} S. Hofmann, G. Lu, D. Mitrea, M. Mitrea and L. Yan,  Hardy spaces associated to non-negative self-adjoint operators satisfying Davies-Gaffney estimates. Mem. Amer. Math. Soc. {\bf214} (2011), no. 1007.


\bibitem{MSTZ} T.  Ma, P.  Stinga, J.  Torrea and C. Zhang,
{Regularity properties of Schr\"odinger operators},
\textit{J. Math. Anal. Appl.}
\textbf{388} (2012), 817--837.




\bibitem{Shen} Z. Shen,
{$L^p$ estimates for Schr\"odinger operators with certain potentials}.
\textit{Ann. Inst. Fourier (Grenoble)}
\textbf{45} (1995), 513--546.


\bibitem{St} E.M. Stein, Topics in Harmonic Analysis Related to the Littlewood-Paley Theory, Princeton
Univ. Press, 1970.


		
\bibitem{U} A. Uchiyama, A remark on Carleson's characterization of BMO, {\it Proc. Amer. Math. Soc},
{\bf 79} (1980), 35-41.


\bibitem{W} J.M. Wilson, Green's theorem and balayage, {\it Michigan Math. J.} {\bf 35} (1988), 21-27.

\end{thebibliography}
\end{document}